\documentclass{article}
\usepackage{amssymb,amsmath}
\usepackage[mathscr]{euscript}
\title{Fixed Points for $E$-Asymptotic Contractions and Boyd-Wong
Type $E$-Contractions in Uniform Spaces\\[0.3cm]}
\author{{Aris Aghanians$^1$,\,\,\,Kamal Fallahi$^1$,\,\,\,Kourosh Nourouzi$^{1}$
\thanks{ Corresponding
author } \thanks {e-mail: nourouzi@kntu.ac.ir; fax: +98 21
22853650}}\\[0.4cm]
{\em $^1$ Department of Mathematics, K. N. Toosi University of Technology,}\\
{\em P.O. Box 16315-1618, Tehran, Iran.}}
\newenvironment{proof}{\noindent{\em{Proof.}}}{$\hfill\square$
\medskip}
\newtheorem{defn}{Definition}
\newtheorem{cor}{Corollary}
\newtheorem{exm}{Example}
\newtheorem{thm}{Theorem}

\newtheorem{rem}{Remark}
\newtheorem{lem}{Lemma}

\begin{document}
\maketitle \begin{abstract}{In this paper we discuss on the fixed
points of asymptotic contractions and Boyd-Wong type contractions
in uniform spaces equipped with an $E$-distance. A new version of
Kirk's fixed point theorem is given for asymptotic
contractions and Boyd-Wong type contractions is investigated in uniform spaces.}
\end{abstract}
\def\thefootnote{ \ }
\footnotetext{{\em} $2010$ Mathematics Subject Classification.
47H10, 54E15, 47H09.}

\section{Introduction and Preliminaries}
In 2003, Kirk \cite{kirk} discussed on the existence of fixed
points for (not necessarily continuous) asymptotic contractions in
complete metric spaces. Jachymski and J\'o\'zwik \cite{jac}
constructed an example to show  that being continuity of the
self-mapping is essential in Kirk's theorem. They also established
a fixed point result for uniformly continuous asymptotic
$\varphi$-contractions in complete metric spaces.\par Motivated by
\cite[Theorem 2.1]{kirk} and \cite[Example 1]{jac} we aim to give
a more general form of \cite[Theorem 2.1]{kirk} in uniform spaces
where the self-mappings are supposed to be continuous. We also
generalize the Boyd-Wong fixed point theorem \cite[Theorem 1]{boy}
to the uniform spaces equipped with an $E$-distance.\par We begin
with some basics in uniform spaces which are needed in this
paper. The reader can find an in-depth discussion in, e.g.,
\cite{wil} and recent results on the fixed points in uniform spaces in \cite{kn}.\par A uniformity on a nonempty set $X$ is a nonempty
collection $\mathscr U$ of subsets of $X\times X$ (called the
entourages of $X$) satisfying the following conditions:
\begin{description}
\item i) Each entourage of $X$ contains the diagonal $\{(x,x):x\in
X\}$; \item ii) $\mathscr U$ is closed under finite intersections;
\item iii) For each entourage $U$ in $\mathscr U$, the set
$\{(x,y):(y,x)\in U\}$ is in $\mathscr U$; \item iv) For each
$U\in\mathscr U$, there exists an entourage $V$ such that
$(x,y),(y,z)\in V$ implies $(x,z)\in U$ for all $x,y,z\in X$;
\item v) $\mathscr U$ contains the supersets of its elements.
\end{description}
If $\mathscr U$ is a uniformity on $X$, then $(X,{\mathscr U})$ (shortly denoted by $X$) is called a uniform space.\par If $d$ is a
metric on a nonempty set $X$, then it induces a uniformity, called
the uniformity induced by the metric $d$, in which the entourages
of $X$ are all the supersets of the sets
$$\big\{(x,y)\in X\times X:d(x,y)<\varepsilon\big\},$$
where $\varepsilon>0$.\par It is well-known that a uniformity
$\mathscr U$ on a nonempty set $X$ is separating if the
intersection of all entourages of $X$ coincides with the diagonal
$\{(x,x):x\in X\}$. In this case $X$ is called a separated
uniform space.\par We next recall some basic concepts about
$E$-distances. For more details and examples the reader is
referred to \cite{aam1}.

\begin{defn}\rm\cite{aam1} Let $X$ be a uniform space. A function $p:X\times X\rightarrow\mathbb{R}^+$
is called an $E$-distance on $X$ if
\begin{description}
\item{\rm i)} for each entourage $U$ in $\mathscr U$, there exists
a $\delta>0$ such that $p(z,x)\leq\delta$ and $p(z,y)\leq\delta$
imply $(x,y)\in U$ for all $x,y,z\in X$; \item{\rm ii)} $p$
satisfies the triangular inequality, i.e.,
$$p(x,y)\leq p(x,z)+p(z,y)\qquad(x,y,z\in X).$$
\end{description}
\end{defn}

If $p$ is an $E$-distance on a uniform space $X$, then a sequence
$\{x_n\}$ in $X$ is said to be $p$-convergent to a point $x\in X$,
denoted by $x_n\stackrel{p}\longrightarrow x$, whenever
$p(x_n,x)\rightarrow0$ as $n\rightarrow\infty$, and $p$-Cauchy
whenever $p(x_m,x_n)\rightarrow0$ as $m,n\rightarrow\infty$. The
uniform space $X$ is called $p$-complete if every $p$-Cauchy
sequence in $X$ is $p$-convergent to some point of $X$.\par The
next lemma contains an important property of $E$-distances on
separated uniform spaces. The proof is straightforward and it is
omitted here.

\begin{lem}\label{1}{\rm\cite{aam1}}
Let $\{x_n\}$ and $\{y_n\}$ be two arbitrary sequences in a
separated uniform space $X$ equipped with an $E$-distance $p$. If
$x_n\stackrel{p}\longrightarrow x$ and
$x_n\stackrel{p}\longrightarrow y$, then $x=y$. In particular,
$p(z,x)=p(z,y)=0$ for some $z\in X$ implies $x=y$.
\end{lem}

Using $E$-distances, $p$-boundedness and $p$-continuity are
defined in uniform spaces.

\begin{defn}\rm\cite{aam1} Let $p$ be an $E$-distance on a uniform space $X$. Then,
\begin{description}
\item{\rm i)} $X$ is called $p$-bounded if
$$\delta_p(X)=\sup\big\{p(x,y):x,y\in X\big\}<\infty.$$
\item{\rm ii)} a mapping $T:X\rightarrow X$ is called
$p$-continuous on $X$ if $x_n\stackrel{p}\longrightarrow x$
implies $Tx_n\stackrel{p}\longrightarrow Tx$ for all sequences
$\{x_n\}$ and all points $x$ in $X$.
\end{description}
\end{defn}

\section{$E$-Asymptotic Contractions}
In this section, we denote by $\Phi$ the class of all functions
$\varphi:\Bbb{R}^{\geq0}\rightarrow\Bbb{R}^{\geq0}$ with the
following properties:
\begin{itemize}
\item $\varphi$ is continuous on $\Bbb{R}^{\geq0}$; \item
$\varphi(t)<t$ for all $t>0$.
\end{itemize}
It is worth mentioning that if $\varphi\in\Phi$, then
$$0\leq\varphi(0)=\lim_{t\rightarrow0^+}\varphi(t)\leq\lim_{t\rightarrow0^+}t=0,$$
that is, $\varphi(0)=0$.\par Following \cite[Definition
2.1]{kirk}, we define $E$-asymptotic contractions.

\begin{defn}\label{defnn}\rm Let $p$ be an $E$-distance on a uniform space $X$.
We say that a mapping $T:X\rightarrow X$ is an $E$-asymptotic
contraction if
\begin{equation}\label{asymp}
p(T^nx,T^ny)\leq\varphi_n\big(p(x,y)\big)\quad\mbox{for all}\
x,y\in X\ \mbox{and}\ n\geq1,
\end{equation}
where $\{\varphi_n\}$ is a sequence of nonnegative functions on
$\Bbb{R}^{\geq0}$ converging uniformly to some $\varphi\in\Phi$ on
the range of $p$.
\end{defn}

If $(X,d)$ is a metric space, then replacing the $E$-distance $p$ by
the metric $d$ in Definition \ref{defnn}, we get the concept of an
asymptotic contraction introduced by Kirk \cite[Definition
2.1]{kirk}. So each asymptotic contraction on a metric space is an
$E$-asymptotic contraction on the uniform space induced by the
metric. But in the next example, we see that the converse is not
generally true.

\begin{exm}\label{exmasymp}\rm Uniformize the set $X=[0,1]$ with the uniformity
induced from the Euclidean metric and put $p(x,y)=y$ for all
$x,y\in X$. It is easily verified that $p$ is an $E$-distance on
$X$. Define $T:X\rightarrow X$ and
$\varphi_1:\Bbb{R}^{\geq0}\rightarrow\Bbb{R}^{\geq0}$ by
$$Tx=\left\{\begin{array}{cc}
0&0\leq x<1\\\\
\displaystyle\frac18&x=1
\end{array}\right.,\qquad\varphi_1(t)=\left\{\begin{array}{cc}
\displaystyle\frac1{16}&0\leq t<1\\\\
\displaystyle\frac18&t\geq1
\end{array}\right.$$
for all $x\in X$ and all $t\geq 0$, and set $\varphi_n=\varphi$
for $n\geq2$, where $\varphi$ is any arbitrary fixed function in
$\Phi$. Clearly, $\varphi_n\rightarrow\varphi$ uniformly on
$\Bbb{R}^{\geq0}$ and $T^n=0$ for all $n\geq2$. To see that $T$ is
an $E$-asymptotic contraction on $X$, it suffices to check
(\ref{asymp}) for $n=1$. To this end, given $x,y\in[0,1]$, if
$y=1$, then we have
$$p(Tx,T1)=T1=\frac18=\varphi_1(1)=\varphi_1\big(p(x,1)\big),$$
and for $0\leq y<1$ we have
$$p(Tx,Ty)=Ty=0\leq\frac1{16}=\varphi_1(y)=\varphi_1\big(p(x,y)\big).$$
But $T$ fails to be an asymptotic contraction on the metric space
$X$ with the functions $\varphi_n$ since
$$\Big|T1-T\frac12\Big|=\frac18>\frac1{16}=\varphi_1\big(\frac12\big)=\varphi_1\Big(\big|1-\frac12\big|\Big).$$
\end{exm}

In the next example, we see that an $E$-asymptotic contraction
need not be $p$-continuous.

\begin{exm}\rm
Let $X$ and $p$ be as in Example \ref{exmasymp}. Define a mapping
$T:X\rightarrow X$ by $Tx=0$ if $0<x\leq1$ and $T0=1$. Note that
$T$ is fixed point free. Now, let $\varphi_1$ be the constant
function $1$ and $\varphi_2=\varphi_3=\cdots=\varphi$, where
$\varphi$ is an arbitrary function in $\Phi$. Then $T$ satisfies
(\ref{asymp}) and since $T0\neq0$, it follows that $T$ fails to be
continuous on.
\end{exm}

\begin{thm}\label{main} Let $p$ be an $E$-distance on a separated
uniform space $X$ such that $X$ is $p$-complete and let
$T:X\rightarrow X$ be a $p$-continuous $E$-asymptotic contraction
for which the functions $\varphi_n$ in Definition \ref{defnn} are
all continuous on $\Bbb{R}^{\geq0}$ for large indices $n$. Then
$T$ has a unique fixed point $u\in X$, and
$T^nx\stackrel{p}\longrightarrow u$ for all $x\in X$.
\end{thm}

\begin{proof} We divide the proof into three steps.\vspace{0.2cm}
\par
{\bf Step 1: $\boldsymbol{p(T^nx,T^ny)\rightarrow0}$ as
$\boldsymbol{n\rightarrow\infty}$ for all $\boldsymbol{x,y\in
X}$.}\par Let $x,y\in X$ be given. Letting $n\rightarrow\infty$ in
(\ref{asymp}), we get
$$0\leq\limsup_{n\rightarrow\infty}p(T^nx,T^ny)\leq\lim_{n\rightarrow\infty}
\varphi_n\big(p(x,y)\big)=\varphi\big(p(x,y)\big)\leq
p(x,y)<\infty.$$ Now, if
$$\limsup_{n\rightarrow\infty}p(T^nx,T^ny)=\varepsilon>0,$$
then there exists a strictly increasing sequence $\{n_k\}$ of
positive integers such that
$p(T^{n_k}x,T^{n_k}y)\rightarrow\varepsilon$, and so by the
continuity of $\varphi$, one obtains
$$\varphi\big(p(T^{n_k}x,T^{n_k}y)\big)\rightarrow\varphi(\varepsilon)<\varepsilon.$$
Therefore  there is an integer $k_0\geq1$ such that
$\varphi(p(T^{n_{k_0}}x,T^{n_{k_0}}y))<\varepsilon$. So
(\ref{asymp}) yields
\begin{eqnarray*}
\varepsilon\hspace{-2mm}&=&\hspace{-2mm}\limsup_{n\rightarrow\infty}p(T^nx,T^ny)\cr\\[-.3cm]
&=&\hspace{-2mm}\limsup_{n\rightarrow\infty}p\big(T^n(T^{n_{k_0}}x),T^n(T^{n_{k_0}}y)\big)\cr\\[-.3cm]
&\leq&\hspace{-2mm}\lim_{n\rightarrow\infty}\varphi_n\big(p(T^{n_{k_0}}x,T^{n_{k_0}}y)\big)\cr\\[-.3cm]
&=&\hspace{-2mm}\varphi\big(p(T^{n_{k_0}}x,T^{n_{k_0}}y)\big)<\varepsilon,
\end{eqnarray*}
which is a contradiction. Hence
$$\limsup_{n\rightarrow\infty}p(T^nx,T^ny)=0.$$
Consequently,
$$0\leq\liminf_{n\rightarrow\infty}p(T^nx,T^ny)\leq\limsup_{n\rightarrow\infty}
p(T^nx,T^ny)=0,$$ that is, $p(T^nx,T^ny)\rightarrow0$.\vspace
{0.2cm}\par {\bf Step 2: The sequence $\boldsymbol{\{T^nx\}}$ is
$\boldsymbol p$-Cauchy for all $\boldsymbol{x\in X}$.}\par Suppose
that $x\in X$ is arbitrary. If $\{T^nx\}$ is not $p$-Cauchy, then
there exist  $\varepsilon>0$ and positive integers $m_k$ and
$n_k$ such that
$$m_k>n_k\geq k\quad\mbox{and}\quad
p(T^{m_k}x,T^{n_k}x)\geq\varepsilon\qquad k=1,2,\ldots\,.$$
Keeping fixed the integer $n_k$ for sufficiently large $k$, say
$k\geq k_0$, and using Step 1, we may assume without loss of
generality that $m_k>n_k$ is the smallest integer with
$p(T^{m_k}x,T^{n_k}x)\geq\varepsilon$, that is,
$$p(T^{m_k-1}x,T^{n_k}x)<\varepsilon.$$
Hence for each $k\geq k_0$, we have
\begin{eqnarray*}
\varepsilon\hspace{-2mm}&\leq&\hspace{-2mm}p(T^{m_k}x,T^{n_k}x)\cr\\[-.4cm]&\leq&\hspace{-2mm}
p(T^{m_k}x,T^{m_k-1}x)+p(T^{m_k-1}x,T^{n_k}x)\cr\\[-.4cm]&<&\hspace{-2mm}
p(T^{m_k-1}x,T^{m_k}x)+\varepsilon.
\end{eqnarray*}
If $k\rightarrow\infty$, since
$p(T^{m_k}x,T^{m_k-1}x)\rightarrow0$, it follows that
$p(T^{m_k}x,T^{n_k}x)\rightarrow\varepsilon$.\par We next show by
induction that
\begin{equation}\label{induct}
\limsup_{k\rightarrow\infty}p(T^{m_k+i}x,T^{n_k+i}x)\geq\varepsilon\qquad
i=1,2,\ldots\,.
\end{equation}
To this end, note first that from Step 1,
\begin{eqnarray*}
\varepsilon\hspace{-2mm}&=&\hspace{-2mm}\lim_{k\rightarrow\infty}p(T^{m_k}x,T^{n_k}x)=\limsup_{k\rightarrow\infty}
p(T^{m_k}x,T^{n_k}x)\cr\\[-.2cm]
&\leq&\hspace{-2mm}\limsup_{k\rightarrow\infty}\Big[p(T^{m_k}x,T^{m_k+1}x)+p(T^{m_k+1}x,T^{n_k+1}x)\cr\\[-.3cm]
&\ &\hspace{-2mm}+\,p(T^{n_k+1}x,T^{n_k}x)\Big]\cr\\[-.2cm]
&\leq&\hspace{-2mm}\limsup_{k\rightarrow\infty}p(T^{m_k}x,T^{m_k+1}x)+
\limsup_{k\rightarrow\infty}p(T^{m_k+1}x,T^{n_k+1}x)\cr\\[-.3cm]
&\ &\hspace{-2mm}+\,\limsup_{k\rightarrow\infty}p(T^{n_k+1}x,T^{n_k}x)\cr\\[-.2cm]
&=&\hspace{-2mm}\limsup_{k\rightarrow\infty}p(T^{m_k+1}x,T^{n_k+1}x),
\end{eqnarray*}
that is, (\ref{induct}) holds for $i=1$. If (\ref{induct}) is true
for an $i$, then
\begin{eqnarray*}
\varepsilon\hspace{-2mm}&\leq&\hspace{-2mm}\limsup_{k\rightarrow\infty}p(T^{m_k+i}x,T^{n_k+i}x)\cr\\[-.3cm]
&\leq&\hspace{-2mm}\limsup_{k\rightarrow\infty}\Big[p(T^{m_k+i}x,T^{m_k+i+1}x)+p(T^{m_k+i+1}x,T^{n_k+i+1}x)\cr\\[-.3cm]
&\ &\hspace{-2mm}+\,p(T^{n_k+i+1}x,T^{n_k+i}x)\Big]\cr\\[-.3cm]
&\leq&\hspace{-2mm}\limsup_{k\rightarrow\infty}p(T^{m_k+i+1}x,T^{n_k+i+1}x).
\end{eqnarray*}
Consequently, we have
$$\begin{array}{cclr}
\varphi(\varepsilon)\hspace{-2mm}&=&\hspace{-2mm}\displaystyle
\lim_{k\rightarrow\infty}\varphi\big(p(T^{m_k}x,T^{n_k}x)\big)&
\mbox{\big(by continuity of $\varphi$\big)}\\\\[-.3cm]
&=&\hspace{-2mm}\displaystyle\lim_{k\rightarrow\infty}
\displaystyle\lim_{i\rightarrow\infty}\varphi_i\big(p(T^{m_k}x,T^{n_k}x)\big)&
\mbox{\big(by pointwise convergence of $\{\varphi_i\}$\big)}\\\\[-.3cm]
&=&\hspace{-2mm}\displaystyle\lim_{i\rightarrow\infty}
\displaystyle\lim_{k\rightarrow\infty}\varphi_i\big(p(T^{m_k}x,T^{n_k}x)\big)&
\mbox{\big(by uniform convergence of $\{\varphi_i\}$\big)}\\\\[-.3cm]
&\geq&\hspace{-2mm}\displaystyle\limsup_{i\rightarrow\infty}
\displaystyle\limsup_{k\rightarrow\infty}p(T^{m_k+i}x,T^{n_k+i}x)&
\mbox{\big(by (\ref{asymp})\big)}\\\\[-.3cm]
&\geq&\hspace{-2mm}\varepsilon,&\mbox{\big(by (\ref{induct})\big)}
\end{array}$$
which is a contradiction. Therefore, $\{T^nx\}$ is
$p$-Cauchy.\vspace {0.2cm}\par {\bf Step 3: Existence and
Uniqueness of the fixed point.}\par Because $X$ is $p$-complete,
it is concluded from Steps 1 and 2 that the family
$\{\{T^nx\}:x\in X\}$ of Picard iterates of $T$ is
$p$-equiconvergent, that is, there exists a $u\in X$ such that
$T^nx\stackrel{p}\longrightarrow u$ for all $x\in X$. In
particular, $T^nu\stackrel{p}\longrightarrow u$. We claim that $u$
is the unique fixed point for $T$. To this end, note first that
since $T$ is $p$-continuous on $X$, it follows that
$T^{n+1}u\stackrel{p}\longrightarrow Tu$, and so, by Lemma
\ref{1}, we have $u=Tu$. And if $v\in X$ is a fixed point for $T$,
then
$$p(u,v)=\lim_{n\rightarrow\infty}p(T^{n}u,T^{n}v)\leq\lim_{n\rightarrow\infty}
\varphi_{n}\big(p(u,v)\big)=\varphi\big(p(u,v)\big),$$ which is
impossible unless $p(u,v)=0$. Similarly  $p(u,u)=0$ and
using Lemma \ref{1}  we get $v=u$.
\end{proof}

It is worth mentioning that the boundedness of some orbit of $T$
is not necessary in Theorem \ref{main} unlike \cite[Theorem
2.1]{kirk} or \cite[Theorem 4.1.15]{aga}.\par As a consequence of
Theorem \ref{main}, we have the following version of \cite[Theorem
3.1]{aam1}.

\begin{cor} Let $p$ be an $E$-distance on a separated
uniform space $X$ such that $X$ is $p$-complete and $p$-bounded
and let a mapping $T:X\rightarrow X$ satisfy
\begin{equation}\label{111}
p(Tx,Ty)\leq\varphi\big(p(x,y)\big)\quad\mbox{for all}\ x,y\in X,
\end{equation}
where $\varphi:\Bbb{R}^{\geq0}\rightarrow\Bbb{R}^{\geq0}$ is
nondecreasing and continuous with $\varphi^n(t)\rightarrow0$ for
all $t>0$. Then $T$ has a unique fixed point $u\in X$, and
$T^nx\stackrel{p}\longrightarrow u$ for all $x\in X$.
\end{cor}

\begin{proof} Note first that $\varphi(0)=0$; for if
$0<t<\varphi(0)$ for some $t$, then the monotonicity of $\varphi$
implies that $0<t<\varphi(0)\leq\varphi^n(t)$ for all $n\geq1$,
which contradicts with $\varphi^n(t)\rightarrow 0$.\par Next,
since $\varphi$ is nondecreasing, it follows that $T$ satisfies
$$p(T^nx,T^ny)\leq\varphi^n\big(p(x,y)\big)\quad\mbox{for all}\
x,y\in X\ \mbox{and}\ n\geq1.$$ Setting $\varphi_n=\varphi^n$ for
each $n\geq1$ in Definition \ref{defnn}, it is seen that
$\{\varphi_n\}$ converges pointwise to the constant function $0$
on $[0,+\infty)$, and since
$$\sup\Big\{\varphi^n\big(p(x,y)\big):x,y\in
X\Big\}\leq\varphi^n\big(\delta_p(X)\big)\rightarrow0,$$ it
follows that $\{\varphi_n\}$ converges uniformly to $0$ on the
range of $p$. Because the constant function $0$ belongs to $\Phi$,
it is concluded that $T$ is an $E$-asymptotic contraction on $X$.
Moreover, $\varphi_n$'s are all continuous on $\Bbb{R}^{\geq0}$
and (\ref{111}) ensures that $T$ is $p$-continuous on $X$.
Consequently, the result follows immediately from Theorem
\ref{main}.
\end{proof}

The next corollary is a partial modification of Kirk's theorem
\cite[Theorem 2.1]{kirk} in uniform spaces. One can find it with
an additional assumption, e.g., in \cite[Theorem 4.1.15]{aga}.

\begin{cor}\label{corasymp} Let $X$ be a complete metric space and $T:X\rightarrow
X$ be a continuous asymptotic contraction for which the functions
$\varphi_n$ in Definition \ref{defnn} are all continuous on
$\Bbb{R}^{\geq0}$ for large indices $n$. Then $T$ has a unique
fixed point $u\in X$, and $T^nx\rightarrow u$ for all $x\in X$.
\end{cor}

\section{Boyd-Wong Type $E$-Contractions}
In this section, we denote by $\Psi$ the class of all functions
$\psi:\Bbb{R}^{\geq0}\rightarrow\Bbb{R}^{\geq0}$ with the
following properties:
\begin{itemize}
\item $\psi$ is upper semicontinuous on $\Bbb{R}^{\geq0}$ from the
right, i.e.,
$$t_n\downarrow
t\geq0\quad\mbox{implies}\quad\limsup_{n\rightarrow\infty}\psi(t_n)\leq\psi(t);$$
\item $\psi(t)<t$ for all $t>0$, and $\psi(0)=0$.
\end{itemize}

It might be interesting for the reader to be mentioned that the
family $\Phi$ defined and used in Section 2 is contained in the
family $\Psi$ but these two families do not coincide. To see this,
consider the function $\psi(t)=0$ if $0\leq t<1$, and
$\psi(t)=\frac12$ if $t\geq1$. Then $\psi$ is upper semicontinuous
from the right  but it is not continuous on $\Bbb{R}^{\geq0}$.
Furthermore, the upper semicontinuity of $\psi$ on
$\Bbb{R}^{\geq0}$ from the right and the condition that
$\psi(t)<t$ for all $t>0$, do not imply that $\psi$ vanishes at
zero in general. In fact, the function
$\psi:\Bbb{R}^{\geq0}\rightarrow\Bbb{R}^{\geq0}$ defined by the
rule
$$\psi(t)=\left\{\begin{array}{cc}
a&t=0\\\\
\displaystyle\frac t2&0<t<1\\\\
\displaystyle\frac1{2t}&t\geq1
\end{array}\right.$$
for all $t\geq0$, where $a$ is an arbitrary positive real number,
confirms this claim.

\begin{thm}\label{boyd} Let $p$ be an $E$-distance on a separated
uniform space $X$ such that $X$ is $p$-complete and let
$T:X\rightarrow X$ satisfy
\begin{equation}\label{boydwong}
p(Tx,Ty)\leq\psi\big(p(x,y)\big)\quad\mbox{for all}\ x,y\in X,
\end{equation}
where $\psi\in\Psi$. Then $T$ has a unique fixed point $u\in X$
and $T^nx\stackrel{p}\longrightarrow u$ for all $x\in X$.
\end{thm}

\begin{proof} We divide the proof into three steps as Theorem
\ref{main}.\vspace {0.2cm}\par {\bf Step 1:
$\boldsymbol{p(T^nx,T^ny)\rightarrow0}$ as
$\boldsymbol{n\rightarrow\infty}$ for all $\boldsymbol{x,y\in
X}$.}\par Let $x,y\in X$ be given. Then for each nonnegative
integer $n$, by the contractive condition (\ref{boydwong}) we have
\begin{equation}\label{boydboyd}
p(T^{n+1}x,T^{n+1}y)\leq\psi\big(p(T^nx,T^ny)\big)\leq
p(T^nx,T^ny).
\end{equation}
Thus, $\{p(T^nx,T^ny)\}$ is a nonincreasing sequence of
nonnegative numbers and so it converges decreasingly to some
$\alpha\geq0$. Letting $n\rightarrow\infty$ in (\ref{boydboyd}),
by the upper semicontinuity of $\psi$ from the right, we get
$$\alpha=\lim_{n\rightarrow\infty}p(T^{n+1}x,T^{n+1}y)\leq
\limsup_{n\rightarrow\infty}\psi\big(p(T^nx,T^ny)\big)\leq\psi(\alpha),$$
which is a contradiction unless $\alpha=0$. Consequently,
$p(T^nx,T^ny)\rightarrow0$.\vspace {0.2cm}\par {\bf Step 2: The
sequence $\boldsymbol{\{T^nx\}}$ is $\boldsymbol{p}$-Cauchy for
all $\boldsymbol{x\in X}$.}\par Let $x\in X$ be arbitrary and
suppose on the contrary that $\{T^nx\}$ is not $p$-Cauchy. Then
similar to the proof of Step 2 of Theorem \ref{main}, it is seen
that there exist an $\varepsilon>0$ and sequences $\{m_k\}$ and
$\{n_k\}$ of positive integers such that $m_k>n_k$ for each $k$
and $p(T^{m_k}x,T^{n_k}x)\rightarrow\varepsilon$. On the other
hand, for each $k$  by (\ref{boydwong}) we have
\begin{eqnarray*}
p(T^{m_k}x,T^{n_k}x)\hspace{-2mm}&\leq&\hspace{-2mm}p(T^{m_k}x,T^{m_k+1}x)+p(T^{m_k+1}x,T^{n_k+1}x)\cr\\[-.3cm]
&\ &\hspace{-2mm}+\ p(T^{n_k+1}x,T^{n_k}x)\cr\\[-.2cm]
&\leq&\hspace{-2mm}p(T^{m_k}x,T^{m_k+1}x)+\psi\big(p(T^{m_k}x,T^{n_k}x)\big)\cr\\[-.3cm]
&\ &\hspace{-2mm}+\ p(T^{n_k+1}x,T^{n_k}x).
\end{eqnarray*}
Letting $k\rightarrow\infty$ and using Step 1 and the upper
semicontinuity of $\psi$ from the right we obtain
\begin{eqnarray*}
\varepsilon\hspace{-2mm}&=&\hspace{-2mm}\lim_{k\rightarrow\infty}p(T^{m_k}x,T^{n_k}x)
=\limsup_{k\rightarrow\infty}p(T^{m_k}x,T^{n_k}x)\cr\\[-.2cm]
&\leq&\limsup_{k\rightarrow\infty}\Big[p(T^{m_k}x,T^{m_k+1}x)+\psi\big(p(T^{m_k}x,T^{n_k}x)\big)\cr\\[-.3cm]
&\ &\hspace{-2mm}+\ p(T^{n_k+1}x,T^{n_k}x)\Big]\cr\\[-.2cm]
&\leq&\hspace{-2mm}\limsup_{k\rightarrow\infty}p(T^{m_k}x,T^{m_k+1}x)+
\limsup_{k\rightarrow\infty}\psi\big(p(T^{m_k}x,T^{n_k}x)\big)\cr\\[-.3cm]
&\ &\hspace{-2mm}+\ \limsup_{k\rightarrow\infty}p(T^{n_k+1}x,T^{n_k}x)\cr\\[-.2cm]
&=&\hspace{-2mm}\limsup_{k\rightarrow\infty}
\psi\big(p(T^{m_k}x,T^{n_k}x)\big)\cr\\[-.2cm]
&\leq&\hspace{-2mm}\psi(\varepsilon),
\end{eqnarray*}
which is a contradiction. Therefore, $\{T^nx\}$ is
$p$-Cauchy.\vspace {0.2cm}\par {\bf Step 3: Existence and
uniqueness of the fixed point.}\par Since $X$ is $p$-complete, it
follows from Steps 1 and 2 that the family $\{\{T^nx\}:x\in X\}$
is $p$-equiconvergent to some $u\in X$. In particular,
$T^nu\stackrel{p}\longrightarrow u$. Since (\ref{boydwong})
implies the $p$-continuity of $T$ on $X$, it follows that
$T^{n+1}u\stackrel{p}\longrightarrow Tu$ and so, by Lemma \ref{1},
we have $u=Tu$, that is, $u$ is a fixed point for $T$. If $v\in X$
is a fixed point for $T$, then
$$p(u,v)=p(Tu,Tv)\leq\psi\big(p(u,v)\big),$$
which is impossible unless $p(u,v)=0$. Similarly $p(u,u)=0$.
Therefore using Lemma \ref{1}  one gets $v=u$.
\end{proof}

As an immediate consequence of Theorem \ref{boydwong} we have the
following fixed point result in metric spaces:

\begin{cor}\label{corboyd}
Let $X$ be a complete metric space and let a mapping
$T:X\rightarrow X$ satisfy
\begin{equation}\label{satisfy}
d(Tx,Ty)\leq\psi\big(d(x,y)\big)\quad\mbox{for all\ }x,y\in X,
\end{equation}
where $\psi\in\Psi$. Then $T$ has a unique fixed point $u\in X$
and $T^nx\rightarrow u$ for all $x\in X$.
\end{cor}

\begin{exm}\rm Let the set $X=[0,1]$ be endowed with the
uniformity induced by the Euclidean metric and define a mapping
$T:X\rightarrow X$ by $Tx=0$ if $0\leq x<1$, and $T1=\frac14$.
Then $T$ does not satisfy (\ref{satisfy}) for any $\psi\in\Psi$
since it is not continuous on $X$. In fact, if $\psi\in\Psi$ is
arbitrary, then
$$\Big|T1-T\frac34\Big|=\frac14>\psi\big(\frac14\big)=\psi\Big(\big|1-\frac34\big|\Big).$$
Now set $p(x,y)=\max\{x,y\}$. Then $p$ is an $E$-distance on $X$
and $T$ satisfies (\ref{boydwong}) for the function
$\psi:\Bbb{R}^{\geq0}\rightarrow\Bbb{R}^{\geq0}$ defined by the
rule $\psi(t)=\frac t4$ for all $t\geq0$. It is easy to check that
this $\psi$ belongs to $\Psi$, and the hypotheses of Theorem
\ref{boyd} are fulfilled.
\end{exm}

\begin{rem}\rm In Theorem \ref{main} (Corollary \ref{corasymp}),
assume that for some index $k$  the function $\varphi_{k}$ belongs
to $\Phi$. Then Theorem \ref{boyd} (Corollary \ref{corboyd})
implies that $T^{k}$ and so $T$ has a unique fixed point $u$ and
$T^{kn}x\stackrel{p}\longrightarrow u$ for all $x\in X$. It
is concluded by the $p$-continuity of $T$ that the family
$\{\{T^nx\}:x\in X\}$ is $p$-equiconvergent to $u$. Hence the
significance of Theorem \ref{main} (Corollary \ref{corasymp}) is
whenever none of $\varphi_n$'s satisfy $\varphi_n(t)<t$ for all
$t>0$, that is, for each $n\geq1$ there exists a $t_n>0$ such
that $\varphi_n(t_n)\geq t_n$.
\end{rem}

\end{document}